\documentclass[12size,hyperref]{article} 
\usepackage{amsmath,amsthm,amssymb,amsfonts,bm}
\usepackage{graphicx}
\usepackage{verbatim}

\newtheorem{corollary}{Corollary}
\newtheorem{proposition}{Proposition}
\newtheorem{lemma}{Lemma}
\newtheorem{theorem}{Theorem}
\newtheorem{definition}{Definition}
\newtheorem{remark}{Remark}

\newtheorem{problem}{Problem}
\newtheorem{conjecture}{Conjecture}

\long\def\symbolfootnote[#1]#2{\begingroup%
\def\thefootnote{\fnsymbol{footnote}}\footnote[#1]{#2}\endgroup}

\begin{document}

\title{Low-dimensional representations of matrix groups and group actions on
\textrm{CAT(0)} spaces and manifolds}
\author{Shengkui Ye \\
National University of Singapore}
\maketitle

\begin{abstract}
We study low-dimensional representations of matrix groups over general
rings, by considering group actions on \textrm{CAT(0)} spaces, spheres and
acyclic manifolds.
\end{abstract}

\section{Introduction\label{section1}}

Low-dimensional representations are studied by many authors, such as
Guralnick and Tiep \cite{gt} (for matrix groups over fields), Potapchik and
Rapinchuk \cite{pr} (for automorphism group of free group), Dokovi\'{c}
and Platonov \cite{dp} (for $\mathrm{Aut}(F_{2})$), Weinberger \cite{sw1}
(for $\mathrm{SL}_{n}(\mathbb{Z})$) and so on. In this article, we study
low-dimensional representations of matrix groups over general rings. Let $R$
be an associative ring with identity and $E_{n}(R)$ $(n\geq 3)$ the group
generated by elementary matrices (cf. Section \ref{sec}). As motivation, we
can consider the following problem.

\begin{problem}
\label{prob}For $n\geq 3,$ is there any nontrivial group homomorphism $%
E_{n}(R)\rightarrow E_{n-1}(R)?$
\end{problem}

Although this is a purely algebraic problem, in general it seems hard to
give an answer in an algebraic way. In this article, we try to answer
Problem \ref{prob} negatively from the point of view of geometric group
theory. The idea is to find a good geometric object on which $E_{n-1}(R)$
acts naturally and nontrivially while $E_{n}(R)$ can only act in a special
way. We study matrix group actions on \textrm{CAT(0)} spaces, spheres and
acyclic manifolds. We prove that for low-dimensional \textrm{CAT(0)} spaces,
a matrix group action always has a global fixed point (cf. Theorem \ref{th3}%
) and that for low-dimensional spheres and acyclic manifolds, a matrix group
action is always trivial (cf. Theorem \ref{th1}). Based on these results, we
show that the low-dimensional representation of matrix groups are quite
constrained (cf. Corollary \ref{correp}) and give a negative answer to
Problem \ref{prob} for some rings $R$ (cf. Corollary \ref{repcor}).
Moreover, these results give generalizations of a result of Farb \cite{fa}
concerning Chevalley groups over commutative rings acting on \textrm{CAT(0)}
spaces and that of Bridson and Vogtmann \cite{BV}, Parwani \cite{Pa} and
Zimmermann \cite{z} concerning the special linear groups $\mathrm{SL}_{n}(%
\mathbb{Z})$ and symplectic groups $\mathrm{Sp}_{2n}(\mathbb{Z})$ acting on
spheres and acyclic manifolds.

We now consider group actions on $\mathrm{CAT(0)}$ spaces. A group $G$ has
Serre's property $\mathrm{FA}$ if any $G$-action on any simplicial tree $T$
has a global fixed point. Recall from Farb \cite{fa} that for an integer $%
n\geq 1$, a group $G$ is said to have \textit{property} $\mathrm{FA}_{n}$ if
any isometric $G$-action on any $n$-dimensional $\mathrm{CAT(0)}$ cell
complex $X$ has a global fixed point. The property $\mathrm{FA}_{1}$ is
Serre's property $\mathrm{FA}.$ If a group $G$ has property $\mathrm{FA}_{n}$
then it has property $\mathrm{FA}_{m}$ for all $m<n.$ Farb \cite{fa} proves
that when a reduced, irreducible root system $\Phi $ has rank $r\geq 2$ and $%
R$ is a finitely generated commutative ring, the elementary subgroup $E(\Phi
,R)$ of Chevalley group $G(\Phi ,R)$ has property $\mathrm{FA}_{r-1}.$ This
gives a generalization of a result obtained by Fukunaga \cite{fu} concerning
groups acting on trees$.$ The group actions on \textrm{CAT(0) }spaces and
property $\mathrm{FA}_{n}$ have also been studied by some other authors. For
example, Bridson \cite{bdm, brid} proves that the mapping class group of a
closed orientable surface of genus $g$ has property $\mathrm{FA}_{g-1}.$ The
group action on \textrm{CAT(0) }spaces of\textrm{\ }automorphism groups of
free groups is studied by Bridson \cite{brid2}. Barnhill \cite{bar}
considers the property $\mathrm{FA}_{n}$ for Coxeter groups.

In this article, we prove the property $\mathrm{FA}_{n}$ for matrix groups
over any ring (not necessary commutative). Without otherwise stated we
assume that a ring is an associative ring with identity. Let $R$ be such a
ring and $n\geq 3$ an integer. Recall the definition of the elementary group
$E_{n}(R)$ generated by elementary matrices and the unitary elementary group
$EU_{2n}(R,\Lambda )$ generated by elementary unitary matrices from Section %
\ref{sec}. When $R$ is a ring of integers in a number field and $n\geq 3,$
the group $E_{n}(R)$ is the special linear group $\mathrm{SL}_{n}(R).$ For
different choices of parameters $\Lambda ,$ the group $EU_{2n}(R,\Lambda )$
contains as special cases the elementary symplectic groups, the elementary
orthogonal groups and the elementary unitary groups.

Watatani \cite{aw} proves that a group with Kazhdan's property $(T)$ has
Serre's property $\mathrm{FA.}$ Ershov and Jaikin-Zapirain \cite{ej} proves
that for a general ring $R$ and an integer $n\geq 3,$ the elementary group $%
E_{n}(R)$ has Kazhdan's property $(T)$. It follows that $E_{n}(R)$ has
Serre's property $\mathrm{FA.}$ Our first result is the following.

\begin{theorem}
\label{th3}Let $R$ be any finitely generated ring and $n\geq 3$ an integer.
Suppose that $E_{n}(R)$ (\textsl{resp.,} $EU_{2n}(R,\Lambda )$) is the
matrix group generated by all elementary matrices (\textsl{resp.,}
elementary unitary matrices). Then the group $E_{n}(R)$ (\textsl{resp.,} $%
EU_{2n}(R,\Lambda ))$ has property $\mathrm{FA}_{n-2}\ $(\textsl{resp., }$%
\mathrm{FA}_{n-1}$).
\end{theorem}

When $R$ is commutative, Theorem \ref{th3} recovers partially the results
obtained by Farb \cite{fa} for Chevalley groups. The dimension in Theorem %
\ref{th3} is sharp, since the group $\mathrm{SL}_{n}(\mathbb{Z}[1/p])$ acts
without a global fixed point on the affine building associated to $\mathrm{SL%
}_{n}(\mathbb{Q}_{p})$ and this building is an $(n-1)$-dimensional,
nonpositively curved simplicial complex.

\begin{remark}
\emph{The property }$\mathrm{FA}_{n-2}$\emph{\ of }$E_{n}(R)$\emph{\
obtained in Theorem \ref{th3} can be viewed as a higher dimensional
generalization of Serre's property }$\mathrm{FA}$\emph{\ for some Kazhdan's
groups. However, it is not clear that every unitary elementary group }$%
EU_{2n}(R,\Lambda )$\emph{\ also has Kazhdan's property }$(T)$\emph{\ (for
property }$(T)$\emph{\ of groups defined by roots, see Ershov,
Jaikin-Zapirain and Kassabov \cite{ejk}).}
\end{remark}

We consider property $\mathrm{FA}_{d}$ for general linear groups $\mathrm{GL}%
_{n}(R)$ over a general ring $R$. For this, we have to introduce notions of $%
K$-groups $K_{1}(R),$ $KU_{1}(R,\Lambda ),$ the stable range $\mathrm{sr}(R)$
and the unitary stable range $\Lambda \mathrm{sr}(R,\Lambda )$ (for details,
see Section \ref{stab}). The stable range is not bigger than most other
famous dimensions of rings, e.g. absolute stable range, $1+$ Krull
dimension, $1+$ maximal spectrum dimension, $1+$ Bass-Serre dimension. When $%
R$ is a Dedekind domain, the stable range $\mathrm{sr}(R)\leq 2.$ When $G$
is a finite group and $\mathbb{Z[}G]$ the integral group ring, the stable
range $\mathrm{sr}(\mathbb{Z[}G])\leq 2.$ The next theorem gives a criterion
when the general linear group $\mathrm{GL}_{n}(R)$ has property $\mathrm{FA}%
_{d}.$

\begin{theorem}
\label{th4}

\begin{enumerate}
\item[(i)] Let $R$ be a finitely generated ring with finite stable range $d=%
\mathrm{sr}(R).$ Suppose that $n\geq d+1$ and the $K$-group $K_{1}(R)$ has
property $\mathrm{FA}_{n-2}$ (e.g. $K_{1}(R)$ is finite). Then the general
linear group $\mathrm{GL}_{n}(R)$ has property $\mathrm{FA}_{n-2}.$

\item[(ii)] Let $(R,\Lambda )$ be a form ring over a finitely generated
associative ring $R$ with a finite $\Lambda $-stable range $d=\Lambda
\mathrm{sr}(R).$ Suppose that $n\geq d+1$ and the $K$-group $%
KU_{1}(R,\Lambda )$ has property $\mathrm{FA}_{n-1}$ (e.g. $KU_{1}(R)$ is
finite). Then the unitary group $U_{2n}(R,\Lambda )$ has property $\mathrm{FA%
}_{n-1}.$
\end{enumerate}
\end{theorem}

Note that the stable range of a ring $A$ of integers in a number field is $2$
and the group $K_{1}(A)$ is $A^{\ast },$ the group of invertible elements in
$A$ (cf. 11.37 in \cite{mag})$.$ According to Theorem \ref{th4}, for any
ring $A$ of integers in a number field with $A^{\ast }$ finite, the general
linear group $\mathrm{GL}_{n}(A)$ has property $\mathrm{FA}_{n-2}$ for $%
n\geq 3.$ Let $G$ be a finite group and $\mathbb{Z}[G]$ the integral group
ring over $G.$ As a corollary to Theorem \ref{th4}, we get a criterion when
the general linear group $\mathrm{GL}_{n}(\mathbb{Z[}G])$ has property $%
\mathrm{FA}_{n-2}.$

\begin{corollary}
\label{cor5}Suppose that $G$ is a finite group with the same number of
irreducible real representations and irreducible rational representations.
Then $K_{1}(\mathbb{Z}[G])$ is finite and for any integer $n\geq 3,$ the
general linear group $\mathrm{GL}_{n}(\mathbb{Z[}G])$ has property $\mathrm{%
FA}_{n-2}.$
\end{corollary}

For example, when $G$ is any symmetric group (cf. page 14 in \cite{ol}), the
general linear group $\mathrm{GL}_{n}(\mathbb{Z[}G])$ has property $\mathrm{%
FA}_{n-2}$ for $n\geq 3.$

We consider the stable elementary groups $E(R)$ and $EU(R,\Lambda )$ acting
on a locally finite \textrm{CAT(0)} cell complex. Recall from Section \ref%
{sec} that the stable elementary group $E(R)$ is a direct limit of $E_{n}(R)$
$(n\geq 2)$ and similarly the stable elementary unitary group $EU(R,\Lambda
) $ is a direct limit of $EU_{2n}(R,\Lambda )$ $(n\geq 2).$ The following
result is obtained:

\begin{proposition}
\label{propcat}Let $R$ be any finitely generated ring. Then any simplicial
isometric action of $E(R)$ or $EU(R,\Lambda )$ on a uniformly locally finite
\textrm{CAT(0)} cell complex is trivial.
\end{proposition}

When $R=\mathbb{Z}$ (so $E(R)=\mathrm{SL}(\mathbb{Z}))$, this is a result
proved by Chatterji and Kassabov (cf. Corollary 4.5 in \cite{ck}).

As the representations of groups with property \textrm{FA}$_{n}$ are quite
constrained, we obtain that for integers $k\geq n$ the elementary group $%
E_{k+1}(R)$ and the unitary elementary group $EU_{2k}(R,\Lambda )$ are
groups of integral $n$-representation type as follows. This theory was
introduced and studied by Bass \cite{ba}. When the ring $R$ in the second
group of Problem \ref{prob} is a field, we have the following.

\begin{corollary}
\label{correp}Let $R$ be a finitely generated ring and an integer $n\geq 2$.
For an integer $k\geq n,$ let $\Gamma $ be the elementary group $E_{k+1}(R)$
or the unitary elementary group $EU_{2k}(R,\Lambda )$ (for $%
EU_{2k}(R,\Lambda ),$ we assume that $k\geq \max \{n,3\}$). Let $\rho
:\Gamma \rightarrow \mathrm{GL}_{n}(K)$ be any representation of degree $n$
over a field $K$. Then

\begin{enumerate}
\item[(i)] the eigenvalues of each of the matrices in $\rho (\Gamma )$ are
integral. In particular they are algebraic integers if the characteristic $%
\mathrm{char}(K)=0$ and are roots of unity if the characteristic $\mathrm{%
char}(K)>0$; and

\item[(ii)] for any algebraically closed field $K,$ there are only finitely
many conjugacy classes of irreducible representations of $\Gamma $ into $%
\mathrm{GL}_{n}(K)$.
\end{enumerate}
\end{corollary}

\bigskip

We now consider group actions on manifolds. The following conjecture is from
Farb and Shalen \cite{fs}, which is related to Zimmer's program (see \cite%
{z3,z87}).

\begin{conjecture}
Any smooth action of a finite-index subgroup of $\mathrm{SL}_{n}(\mathbb{Z})$%
, where $n>2$, on a $r$-dimensional compact manifold $M$ factors through a
finite group action if $r<n-1$.
\end{conjecture}

Parwani \cite{Pa} considers this conjecture for the group $\mathrm{SL}_{n}(%
\mathbb{Z})$ itself and $M$ is a sphere. The idea is to use the theory of
compact transformation groups to show that some sufficiently large finite
subgroups cannot act effectively on $M$, and then to use the Margulis
finiteness theorem to show that any $\mathrm{SL}_{n}(\mathbb{Z})$-action on $%
M$ must be finite. Such techniques are also used several times by many other
authors, e.g. the proof of trivial actions of $\mathrm{SL}_{n}(\mathbb{Z})$
on tori by Weinberger in \cite{sw1}, the proof of the trivial action of $%
\mathrm{SL}(\mathbb{Z})$ on compact manifolds by Weinberger in \cite{sw}
(Proposition 1), the proof of trivial actions of $\mathrm{SL}_{n}(\mathbb{Z}%
) $ on small finite sets by Chatterji and Kassabov in \cite{ck} (Lemma 4.2)
and so on. Zimmermann \cite{z3} actually proves that any smooth action of $%
\mathrm{SL}_{n}(\mathbb{Z})$ on small spheres is trivial. It is natural to
consider other kinds of group actions on compact manifolds. Zimmermann \cite%
{z10} proves a similar trivial action of the symplectic group $\mathrm{Sp}%
_{2n}(\mathbb{Z}).$ The group action of $\mathrm{Aut}(F_{n}),$ the
automorphism group of a free group, on spheres and acyclic manifolds is
considered by Bridson and Vogtmann \cite{BV} and similar trivial-action
results are obtained. More precisely, they show that for $n\geq 3$ and $%
d<n-1,$ any action of the special automorphism group $\mathrm{SAut}(F_{n})$
by homeomorphisms on a generalized $d$-sphere over $\mathbb{Z}_{2}$ or a $%
(d+1)$-dimensional $\mathbb{Z}_{2}$-acyclic homology manifold over $\mathbb{Z%
}_{2}$ is trivial. Hence the group $\mathrm{Aut}(F_{n})$ can act only via
the determinant map $\det :\mathrm{Aut}(F_{n})\rightarrow \mathbb{Z}_{2}.$
In this article, we notice that the Margulis finiteness theorem is not
necessary for such problem. Actually, we get a much more general result for
the actions of matrix groups over any general ring, as follows.

\begin{theorem}
\label{th1}Let $R$ be any ring and $n\geq 3$ be an integer. Suppose that $%
E_{n}(R)$ (\textsl{resp.} $EU_{2n}(R,\Lambda )$) is the matrix group
generated by all elementary matrices (\textsl{resp.} elementary unitary
matrices). Then we have that

\begin{enumerate}
\item[(a)(i)] for an integer $d\leq n-2,$ any action of $E_{n}(R)$ by
homeomorphisms on a generalized $d$-sphere over $\mathbb{Z}_{2}$ is trivial;

\item[(ii)] for an integer $d\leq n-1,$ any action of $E_{n}(R)$ by
homeomorphisms on a $d$-dimensional $\mathbb{Z}_{2}$-acyclic homology
manifold (\textsl{i.e.} has the $\mathbb{Z}_{2}$-homology of a point) is
trivial.

\item[(b)(i)] for an integer $d\leq n-2$ when $n$ is even or $d\leq n-3$
when $n$ is odd, any action of $E_{n}(R)$ by homeomorphisms on a generalized
$d$-sphere over $\mathbb{Z}_{3}$ is trivial;

\item[(ii)] for an integer $d\leq n-1$ when $n$ is even or $d\leq n-2$ when $%
n$ is odd$,$ any action of $E_{n}(R)$ by homeomorphisms on a $d$-dimensional
$\mathbb{Z}_{3}$-acyclic homology manifold (\textsl{i.e.} has the $\mathbb{Z}%
_{3}$-homology of a point) is trivial.

\item[(c)] The statements (a) and (b) also hold for $EU_{2n}(R,\Lambda )$
instead of $E_{n}(R).$
\end{enumerate}
\end{theorem}

When the ring $R=\mathbb{Z}$ and $E_{n}(R)=\mathrm{SL}_{n}(\mathbb{Z}),$ the
above theorem recovers the results obtained by Bridson and Vogtmann \cite{BV}%
, Parwani \cite{Pa} and Zimmermann \cite{z}. The dimensions in (a) and those
in (b) with even $n$ of Theorem \ref{th1} are sharp, since the group $%
\mathrm{SL}_{n}(\mathbb{Z})=E_{n}(\mathbb{Z})$ ($n\geq 3$) can act
nontrivially on the standard sphere $S^{n-1}$ and the Euclidean space $%
\mathbb{R}^{n}.$

If the parameter $\Lambda $ in the definition of form ring $(R,\Lambda )$
contains the identity $1\in R$, we can get an improvement of Theorem \ref%
{th1} as following.

\begin{theorem}
\label{th2}Let $(R,\Lambda )$ be a form ring. Suppose that $1\in \Lambda .$
Then we have that

\begin{enumerate}
\item[(i)] for an integer $d\leq 2n-2$, any action of $EU_{2n}(R,\Lambda )$
by homeomorphisms on a generalized $d$-sphere over $\mathbb{Z}_{3}$ is
trivial;

\item[(ii)] for an integer $d\leq 2n-1,$ any action of $EU_{2n}(R,\Lambda )$
by homeomorphisms on a $d$-dimensional $\mathbb{Z}_{3}$-acyclic homology
manifold (\textsl{i.e.} has the $\mathbb{Z}_{3}$-homology of a point) is
trivial.
\end{enumerate}
\end{theorem}

When the ring $R=\mathbb{Z}$ and $EU_{2n}(R,\Lambda )=\mathrm{Sp}_{2n}(%
\mathbb{Z}),$ the above theorem recovers a result obtained by Zimmermann in
\cite{z10}. Considering the nontrivial actions of the symplectic group $%
\mathrm{Sp}_{2n}(\mathbb{Z})$ on $S^{2n-1}$ and $\mathbb{R}^{2n},$ we see
that the dimensions in Theorem \ref{th2} are sharp.

As an easy corollary of Theorem \ref{th1} and Theorem \ref{th2}, we have a
negative answer to Problem \ref{prob} when $R$ is a subring of the real
numbers $\mathbb{R}$.

\begin{corollary}
\label{repcor}Let $R$ be a general ring and $S$ a commutative ring. Assume
that $A$ is a subring of the real numbers $\mathbb{R}$ and $n\geq 3.$ Then

\begin{enumerate}
\item[(i)] any group homomorphism
\begin{equation*}
E_{n}(R)\rightarrow E_{n-1}(A)
\end{equation*}%
is trivial;

\item[(ii)] any group homomorphism%
\begin{equation*}
\mathrm{Sp}_{2n}(S)\rightarrow \mathrm{Sp}_{2(n-1)}(A)
\end{equation*}%
is trivial.
\end{enumerate}
\end{corollary}

As an easy corollary of Theorem \ref{th1} and Theorem \ref{th2}, we see that
the group $E(R)$ or $EU(R,\Lambda )$ cannot act nontrivially by
homeomorphisms on any generalized $d$-sphere or $\mathbb{Z}_{2}$-acyclic
homology manifold. Actually, on any compact manifold, the following theorem
shows that there are no nontrivial actions of $E(R)$ and $EU(R,\Lambda ).$

\begin{theorem}
\label{thcor}Let $R$ be any ring, $E(R)$ and $EU(R,\Lambda )$ the stable
elementary and unitary elementary groups. Then the group $E(R)$ or $%
EU(R,\Lambda )$ does not act topologically, nontrivially, on any compact
manifold, or indeed on any manifold whose homology with coefficients in a
field of positive characteristic is finitely generated.
\end{theorem}

When $R=\mathbb{Z}$ and $E(R)=\mathrm{SL}(\mathbb{Z}),$ Theorem \ref{thcor}
is Proposition 1 in \cite{sw}.

\begin{remark}
\emph{Let }$R$\emph{\ be any ring and }$n\geq 3$\emph{. The author believe
that the same results in this article hold as well for the Steinberg group }$%
St_{n}(R)$\emph{\ (resp., the unitary Steinberg group }$USt_{n}(R,\Lambda )$%
\emph{) instead of the elementary group }$E_{n}(R)$\emph{\ (resp., the
unitary elementary group }$EU_{n}(R,\Lambda )$\emph{). This is because in
the proofs(cf. last section), only commutator formulas are used and these
formulas hold as well for Steinberg groups (resp., unitary Steinberg groups).%
}
\end{remark}

In Section 2, we give some basic facts of $\mathrm{CAT(0)}$ spaces and
homology manifolds. In Section 3, we introduce the notions of elementary
groups $E_{n}(R)$, $EU_{2n}(R,\Lambda )$, Steinberg groups, the algebraic $K$%
-groups $K_{1}$, $KU_{1}$ and the stable ranges. The results in the
introduction will be proved in Section 4.

\section{Notations and basic facts}

\noindent \textbf{Standard assumptions.} In this article, we assume that all
the rings are associative rings with identity. All the \textrm{CAT(0)}
spaces are complete and all actions on them are isometric and semisimple.
When we talk about groups with properties $\mathrm{FA}_{n},$ we always
assume that the groups are finitely generated.

\subsection{CAT(0) spaces and property $\mathrm{FA}_{n}$}

Let $(X,d_{X})$ be a geodesic metric space. For three points $x,y,z\in X,$
the geodesic triangle $\Delta (x,y,z)$ consists of the three vertices $x,y,z$
and the three geodesics $[x,y],[y,z]$ and $[z,x].$ Let $\mathbb{R}^{2}$ be
the Euclidean plane with the standard distance $d_{\mathbb{R}^{2}}$ and $%
\bar{\Delta}$ a triangle in $\mathbb{R}^{2}$ with the same edge lengths as $%
\Delta $. Denote by $\varphi :\Delta \rightarrow \bar{\Delta}$ the map
sending each edge of $\Delta $ to the corresponding edge of $\bar{\Delta}.$
The space $X$ is called a \textrm{CAT(0)} space if for any triangle $\Delta $
and two elements $a,b\in \Delta ,$ we have the inequality
\begin{equation*}
d_{X}(a,b)\leq d_{\mathbb{R}^{2}}(\varphi (a),\varphi (b)).
\end{equation*}%
The typical examples of \textrm{CAT(0)} spaces include simplicial trees,
hyperbolic spaces, products of \textrm{CAT(0) }spaces and so on. From now
on, we assume that $X$ is a complete \textrm{CAT(0)} space. Denote by
\textrm{Isom}$(X)$ the isometry group of $X.$ For any $g\in $ \textrm{Isom}$%
(X)$\textrm{, }let
\begin{equation*}
\mathrm{Minset}(g)=\{x\in X:d(x,gx)\leq d(y,gy)\text{ for any }y\in X\}
\end{equation*}%
and let $\tau (g)=\inf\nolimits_{x\in X}d(x,gx)$ be the translation length
of $g.$ When the fixed-point set $\mathrm{Fix}(g)\neq \emptyset ,$ we call $%
g $ elliptic. When $\mathrm{Minset}(g)\neq \emptyset $ and $d_{X}(x,gx)=\tau
(g)>0$ for any $x\in \mathrm{Minset}(g),$ we call $g$ hyperbolic. The group
element $g$ is called semisimple if the minimal set $\mathrm{Minset}(g)$ is
not empty, i.e. it is either elliptic or hyperbolic. By a \textrm{CAT(0)}
complex, we mean a \textrm{CAT(0)} cell complex of piecewise constant
curvature with only finitely many isometry types of cells. For more details
on \textrm{CAT(0) }spaces, see the book of Bridson and Haefliger \cite{bh}.

The following definition of property $\mathrm{FA}_{n}$ and strong $\mathrm{FA%
}_{n}$ property were given by Farb \cite{fa} as a generalization of Serre's
property $\mathrm{FA}.$

\begin{definition}
\emph{Let }$n\geq 1$\emph{\ be an integer. A group }$\Gamma $\emph{\ is said
to have property }$\mathrm{FA}_{n}$\emph{\ if any isometric }$\Gamma $\emph{%
-action on any }$n$\emph{-dimensional, \textrm{CAT(0)} cell complex }$X$%
\emph{\ has a global fixed point. A group }$\Gamma $\emph{\ is said to have
strong }$\mathrm{FA}_{n}$\emph{\ property if any }$\Gamma $\emph{-action on
a complete \textrm{CAT(0)} space }$X$\emph{\ satisfying the following two
properties has a global fixed point.}

\begin{enumerate}
\item[(i)] $n$\emph{-dimensionality: The reduced homology group }$\tilde{H}%
_{n}(Y;\mathbb{Z})=0$\emph{\ for all open subsets }$Y\subseteq X.$

\item[(ii)] \emph{Semisimplicity: The action of }$\Gamma $\emph{\ on }$X$%
\emph{\ is semisimple, i.e., the translation length of each }$g\in \Gamma $%
\emph{\ is realized by some }$x\in X.$
\end{enumerate}
\end{definition}

When $n=1,$ the property $\mathrm{FA}_{1}$ corresponds with Serre's property
$\mathrm{FA}.$ Since any isometric action on a \textrm{CAT(0)} cell complex
must be semisimple (cf. page 231 in \cite{bh}), we see that strong $\mathrm{%
FA}_{n}$ implies $\mathrm{FA}_{n}.$ The following lemma contains some
general facts on $\mathrm{FA}_{n}$ (see pages 1578-1579 in \cite{fa} for
more details).

\begin{lemma}
\label{lemfarb}The following properties hold:

\begin{enumerate}
\item[(1)] If $G$ has property $\mathrm{FA}_{n}$ then $G$ has $\mathrm{FA}%
_{m}$ for all $m\leq n.$

\item[(2)] If $G$ has $\mathrm{FA}_{n}$ then so does every quotient group of
$G.$

\item[(3)] Let $H$ be a normal subgroup of $G.$ If $H$ and $G/H$ have $%
\mathrm{FA}_{n}$ then so does $G.$

\item[(4)] If some finite index subgroup $H$ of $G$ has $\mathrm{FA}_{n},$
then so does $G.$
\end{enumerate}
\end{lemma}

\subsection{Homology manifolds and Smith theory}

Since the fixed-point set of a finite-period homeomorphism of a manifold is
not necessary a manifold any more, we are working with generalized
manifolds. All homology groups in this subsection are Borel-Moore homology
with compact supports and coefficients in a sheaf $\mathcal{A}$ of modules
over a principal ideal domain $L.$ All the concepts below are from Bredon's
book \cite{br}. Let $X$ be a locally finite CW-complex and $\mathcal{A}$ be
the constant sheaf $X\times L$ (simply denoted by $L$). The homology groups $%
H_{\ast }^{c}(X)$ of $X$ are isomorphic to the singular homology groups with
coefficients in $L$ (cf. page 279 in \cite{br}).

Let $L$ be the integers $\mathbb{Z}$ or the finite field $\mathbb{Z}_{p}$
for a prime $p.$ The following definition is from page 329 in \cite{br} (see
also Definition 4.1 of \cite{BV}).

\begin{definition}
\emph{An }$m$\emph{-dimensional homology manifold over }$L$\emph{\ (denoted }%
$m$\emph{-hm}$_{L}$\emph{) is a locally compact Hausdorff space }$X$\emph{\
with finite homological dimension over }$L$\emph{\ that has the local
homology properties of a manifold of dimension }$m$\emph{.}
\end{definition}

The homology spheres and homology acyclic manifolds are defined as follows
(cf. Definition 4.2 and 4.3 of \cite{BV}).

\begin{definition}
\emph{Let }$S^{n}$\emph{\ be the standard }$n$\emph{-dimensional sphere. If }%
$X$\emph{\ is an }$m$\emph{-hm}$_{L}$\emph{\ and }$H_{\ast }^{c}(X;L)\cong
H_{\ast }^{c}(S^{m};L)$\emph{\ then }$X$\emph{\ is called a generalized }$m$%
\emph{-sphere over }$L.$\emph{\ If }$X$\emph{\ is an }$m$\emph{-hm}$_{L}$%
\emph{\ with }$H_{0}^{c}(X;L)=L$\emph{\ and }$H_{k}^{c}(X;L)=0$\emph{\
for }$k>0,$\emph{\ then }$X$\emph{\ is said to be }$L$\emph{-acyclic.}
\end{definition}

The following "global" Smith theorem was originally proved by P.A. Smith (%
\cite{sm1}, \cite{sm2}). Here we follow the exposition in Bredon's book \cite%
{br}. The following lemma is a combination of Corollary 19.8 and Corollary
19.9 (page 144) in \cite{br} (see also Theorem 4.5 in \cite{BV}).

\begin{lemma}
\label{locsmith}Let $p$ be a prime and $X$ be a locally compact Hausdorff
space of finite dimension over $\mathbb{Z}_{p}.$ Suppose that $\mathbb{Z}%
_{p} $ acts on $X$ with fixed-point set $F.$

\begin{enumerate}
\item[(i)] If $H_{\ast }^{c}(X;\mathbb{Z}_{p})\cong H_{\ast }^{c}(S^{m};%
\mathbb{Z}_{p}),$ then $H_{\ast }^{c}(F;\mathbb{Z}_{p})\cong H_{\ast
}^{c}(S^{r};\mathbb{Z}_{p})$ for some $r$ with $-1\leq r\leq m.$ If $p$ is
odd, then $r-m$ is even.

\item[(ii)] If $X$ is $\mathbb{Z}_{p}$-acyclic, then $F$ is $\mathbb{Z}_{p}$%
-acyclic (in particular nonempty and connected).
\end{enumerate}
\end{lemma}

\section{Elementary groups and $K$-theory}

\subsection{Elementary groups and Steinberg groups$\label{sec}$}

In this subsection, we briefly recall the definitions of the elementary
subgroups $E_{n}(R)$ of the general linear group $\mathrm{GL}_{n}(R)$, the
unitary elementary subgroup $EU_{2n}(R,\Lambda )$ of the unitary group $%
U_{2n}(R,\Lambda )$ and the Steinberg groups $\mathrm{St}_{n}(R)$. For more
details, see the book of Magurn \cite{mag}, the book of Hahn and O'Meara
\cite{ho} and the book of Bak \cite{bak}. We define the groups $\mathrm{GL}%
_{n}(R)$ and $E_{n}(R)$ first. Let $R$ be an associative ring with identity
and $n\geq 2$ be an integer. The general linear group $\mathrm{GL}_{n}(R)$
is the group of all $n\times n$ invertible matrices with entries in $R$. For
an element $r\in R$ and any integers $i,j$ such that $1\leq i\neq j\leq n,$
denote by $e_{ij}(r)$ the elementary $n\times n$ matrix with $1$ in the
diagonal positions and $r$ in the $(i,j)$-th position and zeros elsewhere.
The group $E_{n}(R)$ is generated by all such $e_{ij}(r),$\textsl{\ i.e. }%
\begin{equation*}
E_{n}(R)=\langle e_{ij}(r)|1\leq i\neq j\leq n,r\in R\rangle .
\end{equation*}
Denote by $I_{n}$ the identity matrix and by $[a,b]$ the commutator $%
aba^{-1}b^{-1}.$

The following lemma displays the commutator formulas for $E_{n}(R)$ (cf.
Lemma 9.4 in \cite{mag}).

\begin{lemma}
\label{ecom}Let $R$ be a ring and $r,s\in R.$ Then for distinct integers $%
i,j,k,l$ with $1\leq i,j,k,l\leq n,$ the following hold:

\begin{enumerate}
\item[(1)] $e_{ij}(r+s)=e_{ij}(r)e_{ij}(s);$

\item[(2)] $[e_{ij}(r),e_{jk}(s)]=e_{ik}(rs);$

\item[(3)] $[e_{ij}(r),e_{kl}(s)]=I_{n}.$
\end{enumerate}
\end{lemma}

By Lemma \ref{ecom}, the group $E_{n}(R)$ is finitely generated when the
ring $R$ is finitely generated. Moreover, when $n\geq 3,$ the group $%
E_{n}(R) $ is normally generated by any elementary matrix $e_{ij}(1).$ We
will use such fact several times in Section 4.

The commutator formulas can be used to define Steinberg group as follows.
For $n\geq 3,$ the Steinberg group $\mathrm{St}_{n}(R)$ is the group
generated by the symbols $\{x_{ij}(r):1\leq i\neq j\leq n,r\in R\}$ subject
to the following relations:

\begin{enumerate}
\item[(St$1$)] $x_{ij}(r+s)=x_{ij}(r)x_{ij}(s);$

\item[(St$2$)] $[e_{ij}(r),e_{jk}(s)]=e_{ik}(rs)$ for $i\neq k;$

\item[(St$3)$] $[e_{ij}(r),e_{kl}(s)]=1$ for $i\neq l,j\neq k.$
\end{enumerate}

\noindent There is an obvious surjection $\mathrm{St}_{n}(R)\rightarrow
E_{n}(R)$ defined by $x_{ij}(r)\longmapsto e_{ij}(r).$

For any ideal $I\vartriangleleft R,$ let $p:R\rightarrow R/I$ be the
quotient map. Then the map $p$ induces a group homomorphism $p_{\ast }:%
\mathrm{St}_{n}(R)\rightarrow \mathrm{St}_{n}(R/I).$ Denote by $\mathrm{St}%
_{n}(R,I)$ (\textsl{resp., }$E_{n}(R,I)$) the subgroup of $\mathrm{St}%
_{n}(R) $ (\textsl{resp., }$E_{n}(R)$) normally generated by elements of the
form $x_{ij}(r)$ (\textsl{resp.,} $e_{ij}(r)$) for $r\in I$ and $1\leq i\neq
j\leq n.$ In fact, $\mathrm{St}_{n}(R,I)$ is the kernel of $p_{\ast }$ (cf.
Lemma 13.18 in Magurn \cite{mag} and its proof). However, $E_{n}(R,I)$ may
not be the kenel of $E_{n}(R)\rightarrow E_{n}(R/I)$ induced by $p.$

We define the groups $U_{2n}(R,\Lambda )$ and $EU_{2n}(R,\Lambda )$ as
follows. Let $R$ be a general ring and assume that an anti-automorphism $%
\ast :x\mapsto x^{\ast }$ is defined on $R$ such that $x^{\ast \ast
}=\varepsilon x\varepsilon ^{\ast }$ for some unit $\varepsilon =\varepsilon
^{-1}$ of $R$ and every $x$ in $R$. It determines an anti-automorphism of
the ring $M_{n}R$ of all $n\times $ $n$ matrices $(x_{ij})$ by $%
(x_{ij})^{\ast }=(x_{ji}^{\ast })$.

Set $R_{\varepsilon }=\{x-x^{\ast }\varepsilon |\,\,x\in R\}$ and $%
R^{\varepsilon }=\{x\in R\,|\,\,x=-x^{\ast }\varepsilon \}$. If some
additive subgroup $\Lambda $ of $(R,+)$ satisfies:{}{}

\begin{enumerate}
\item[(i)] $r^{\ast }\Lambda r\subset \Lambda $ for all $r\in R$;

\item[(ii)] $R_{\varepsilon }\subset \Lambda \subset R^{\varepsilon },$
\end{enumerate}

\noindent we will call $\Lambda $ a form and $(\Lambda ,\ast ,\varepsilon )$
a form parameter on $R$. Usually $(R,\Lambda )$ is called a form ring. Let $%
\Lambda _{n}=\{(a_{ij})\in M_{n}R|\,\,a_{ij}=-a_{ji}^{\ast }\varepsilon $
for $i\neq j$ and $a_{ii}\in \Lambda \}.$

As in \cite{bak}, for an integer $n\geq 1$ we define the unitary group%
\newline
\begin{equation*}
U_{2n}(R,\Lambda )=\left\{ \left(
\begin{array}{ll}
\alpha & \beta \\
\gamma & \delta%
\end{array}%
\right) \in GL_{2n}R\,|\,\,\alpha ^{\ast }\delta +\gamma ^{\ast }\varepsilon
\beta =I_{n},\,\,\,\alpha ^{\ast }\gamma ,\,\,\,\beta ^{\ast }\delta \in
\Lambda _{n}\right\} .
\end{equation*}%
Sometimes, the unitary group $U_{2n}(R,\Lambda )$ is also called the
quadratic group \cite{bak} or the pseudo-orthogonal group \cite{vh}.

It can be easily seen that the inverse of a unitary matrix has the form
\begin{equation*}
\left(
\begin{array}{ll}
\alpha & \beta \\
\gamma & \delta%
\end{array}%
\right) ^{-1}=\left(
\begin{array}{ll}
\varepsilon ^{\ast }\delta ^{\ast }\varepsilon & \varepsilon ^{\ast }\beta
^{\ast } \\
\gamma ^{\ast }\varepsilon & \alpha ^{\ast }%
\end{array}%
\right) .
\end{equation*}%
The unitary group $U_{2n}(R,\Lambda )$ has many important special cases, as
follows.

\begin{itemize}
\item When $\Lambda =R$, $U_{2n}(R,\Lambda )$ is the symplectic group. This
can only happen when $\varepsilon =-1$ and $\ast =\mathrm{id}_{R}$ ($R$ is
commutative) the trivial anti-automorphism.

\item When $\Lambda =0$, $U_{2n}(R,\Lambda )$ is the ordinary orthogonal
group. This can only happen when $\varepsilon =1$ and $\ast =\mathrm{id}%
_{R}\ $($R$ is commutative) as well.

\item When $\Lambda =R^{\varepsilon }$ and $\ast \neq \mathrm{id}_{R}$, $%
U_{2n}(R,\Lambda )$ is the classical unitary group
\begin{equation*}
U_{2n}=\{A\in GL_{2n}R|\,\ A^{\ast }\varphi _{n}A=\varphi _{n}\},
\end{equation*}%
where
\begin{equation*}
\varphi _{n}=\left(
\begin{array}{cc}
0 & I_{n} \\
\varepsilon I_{n} & 0%
\end{array}%
\right) .
\end{equation*}
\end{itemize}

Let $E_{ij}$ denote the $n\times n$ matrix with $1$ in the $(i,j)$-th
position and zeros elsewhere. Then $e_{ij}(a)=I_{n}+aE_{ij}$ is an
elementary matrix, where $I_{n}$ is the identity matrix of size $n$. With $n$
fixed, for any integer $1\leq k\leq 2n,$ set $\sigma k=k+n$ if $k\leq n$ and
$\sigma k=k-n$ if $k>n$. For $a\in R$ and $1\leq i\neq j\leq 2n,$ we define
the elementary unitary matrices $\rho _{i,\sigma i}(a)$ and $\rho _{ij}(a)$
with $j\neq \sigma i$ as follows:

\begin{itemize}
\item $\rho _{i,\sigma i}(a)=I_{2n}+aE_{i,\sigma i}$ with $a\in \Lambda $
when $n+1\leq i$ and $a^{\ast }\in \Lambda $ when $i\leq n$;

\item $\rho _{ij}(a)=\rho _{\sigma j,\sigma i}(-a^{\prime
})=I_{2n}+aE_{ij}-a^{\prime }E_{\sigma j,\sigma i}$ with $a^{\prime
}=a^{\ast }$ when $i,j\leq n$; $a^{\prime }={\varepsilon ^{\ast }a^{\ast }}$
when $i\leq n<j$; $a^{\prime }=a^{\ast }\varepsilon $ when $j\leq n<i$; and $%
a^{\prime }=\varepsilon ^{\ast }a^{\ast }\varepsilon $ when $n+1\leq i,j$.
\end{itemize}

The following lemma displays the commutator formulas for $EU_{n}(R,\Lambda )$
(cf. Lemma 2.1 in \cite{vh})

\begin{lemma}
\label{2.1} The following identities hold for elementary unitary matrices $%
(1\leq i\neq j\leq 2n):$

\emph{1).} $\rho_{ij}(a+b) = \rho_{ij}(a) \rho_{ij}(b)$;

\emph{2). }$[\rho_{ij}(a), \rho_{jk}(b)] = \rho_{ik}(ab)$ when $i,\,j, \,k,
\,\sigma i, \,\sigma j, \,\sigma k$ are distinct;

\emph{3). }$[\rho_{ij}(a), \rho_{j,\sigma i}(b)] = \rho_{i, \sigma i}(ab -
c) $ when $j \neq \sigma i,$ where $c = b^* a^* \epsilon$ when $n + 1 \leq i$
and $c = \epsilon^* b^* a^*$ when $i \leq n;$

\emph{4).} $[\rho _{ij}(a),\rho _{j,\sigma j}(b)]=\rho _{i,\sigma j}(ab)\rho
_{i,\sigma i}(c)$ when $j\neq \sigma i,$ where $b^{\ast }\in \Lambda $ and $%
c=aba^{\ast }$ when $i,j\leq n,$ $b^{\ast }\in \Lambda $ and $c=aba^{\ast
}\epsilon $ when $j\leq n<i,$ $b\in \Lambda $ and $c=-ab^{\ast }a^{\ast }$
when $i\leq n<j,$ $b\in \Lambda $ and $c=-ab^{\ast }a^{\ast }\epsilon $ when
$n+1\leq i,j.$
\end{lemma}

When the ring $R$ is finitely generated and $\Lambda /R_{\varepsilon }$ is a
finitely generated $R$-module by right multiplications, the above commutator
formulas show that $E_{n}(R,\Lambda )$ is finitely generated (cf. \cite{ho},
Section 9.2B). Our later discussions will base on the following lemma.

\begin{lemma}
\label{prop}Let $R$ be a ring and assume that the characteristic of $R$ is
not $2.$ For two integers $i,j$ such that $1\leq i\neq j\leq n,$ let $A_{ij}$
be the diagonal matrix whose $(i,i)$-th and $(j,j)$-th entries are $-1$ and
other diagonal entries are $1.$ Then the subgroup generated by the elements%
\begin{equation*}
A_{12},A_{23},\ldots ,A_{n-1,n}
\end{equation*}%
in $E_{n}(R)$ is isomorphic to the abelian group $\mathbb{Z}_{2}^{n-1}=%
\mathbb{Z}_{2}\times \mathbb{Z}_{2}\times \cdots \times \mathbb{Z}_{2}$,
i.e. $n-1$ copies of groups of two elements.
\end{lemma}

\begin{proof}
First note that there is an equality
\begin{equation*}
A_{12}=e_{12}(1)e_{21}(-1)e_{12}(1)e_{12}(1)e_{21}(-1)e_{12}(1),
\end{equation*}%
which shows $A_{12}$ in an element of $E_{n}(R).$ Similar arguments show
that all other elements $A_{i,i+1}$ are also in $E_{n}(R).$ It is not hard
to see that the elements are pairwise commutative and the subgroup generated
is isomorphic to the $n-1$ copies of $\mathbb{Z}_{2}$.
\end{proof}

\begin{theorem}
\label{normal}Let $R$ be a ring with identity and $n\geq 3$ an integer.
Suppose that the characteristic of $R$ is not $2$ and $G$ is a normal
subgroup in $E_{n}(R)$ containing a noncentral element in the subgroup
generated $\mathbb{Z}_{2}^{n-1}$ by $A_{12},A_{23},\ldots ,A_{n-1,n}$ in
Lemma \ref{prop}. Then $G$ contains $E_{n}(R,2R)$ as a normal subgroup,%
\textsl{\ i.e.,}
\begin{equation*}
E_{n}(R,2R)\unlhd G.
\end{equation*}
\end{theorem}

\begin{proof}
Let $A\in G$ be a noncentral element of $E_{n}(R)$ in $\mathbb{Z}_{2}^{n-1},$
the subgroup generated by $A_{12},A_{23},\ldots ,A_{n-1,n}.$ In other words,
$A\neq \mathrm{diag}(1,\ldots ,1),$ $\mathrm{diag}(-1,\ldots ,-1).$ Without
loss of generality, we assume that the first three diagonal entries of $A$
are $1,-1,-1$ in order. Then for any element $r\in R,$ we have that the
matrix $e_{12}(2r)$ is the product
\begin{equation*}
e_{12}(r)\cdot A\cdot e_{12}(-r)\cdot A^{-1},
\end{equation*}%
which is an element in $G.$ By the commutator formulas in Lemma \ref{ecom},
we see that for any two integers $i,j$ with $1\leq i\neq j\leq n,$ the
matrix $e_{ij}(2r)\in G.$ This shows that $E_{n}(R,2R)$ is a normal subgroup
of $G.$
\end{proof}

\subsection{$K$-theory and stable ranges\label{stab}}

In this subsection, we briefly recall the definitions of algebraic $K_{1}$
and unitary $KU_{1}$ groups. The standard references are also the textbook
of Magurn \cite{mag} (for $K_{1}$), the book of Hahn and O'Meara \cite{ho}
and the book of Bak \cite{bak} (for $KU_{1}$).

We define $K_{1}$ first. For a ring $R$, let $\mathrm{GL}_{n}(R)\rightarrow
\mathrm{GL}_{n+1}(R)$ be the inclusion defined by
\begin{equation*}
A\longmapsto
\begin{pmatrix}
A & 0 \\
0 & 1%
\end{pmatrix}%
.
\end{equation*}%
The group $\mathrm{GL}(R)$ is defined to be the direct limit of%
\begin{equation*}
\mathrm{GL}_{1}(R)\subset \mathrm{GL}_{2}(R)\subset \cdots \subset \mathrm{GL%
}_{n}(R)\subset \cdots .
\end{equation*}%
Similarly, the group $E(R)$ is the direct limit of%
\begin{equation*}
E_{2}(R)\subset E_{3}(R)\subset \cdots \subset E_{n}(R)\subset \cdots .
\end{equation*}%
According to the Whitehead lemma, the group $E(R)$ is normal in $\mathrm{GL}%
(R).$ The $K$-theory group $K_{1}(R)$ is defined as $\mathrm{GL}(R)/E(R).$

The stable range $\mathrm{sr}(R)$ is defined as follows. Let $n$ be a
positive integer and $R^{n}$ the free $R$-module of rank $n$ with standard
basis. A vector $(a_{1},\ldots ,a_{n})$ in $R^{n}$ is called \emph{right
unimodular} if there are elements $b_{1},\ldots ,b_{n}\in R$ such that $%
a_{1}b_{1}+\cdots +a_{n}b_{n}=1$. The \emph{stable range condition} $\mathrm{%
sr}_{m}$ says that if $(a_{1},\ldots ,a_{m+1})$ is a right unimodular vector
then there exist elements $b_{1},\ldots ,b_{m}\in R$ such that $%
(a_{1}+a_{m+1}b_{1},\ldots ,a_{m}+a_{m+1}b_{m})$ is right unimodular. It
follows easily that $\mathrm{sr}_{m}\Rightarrow \mathrm{sr}_{n}$ for any $%
n\geq m$. The \emph{stable range} $\mathrm{sr}(R)$ of $R$ is the smallest
number $m$ such that $\mathrm{sr}_{m}$ holds. If $R$ is commutative, the
Krull dimension $\mathrm{Kdim}(R)$ of $R$ is the number of steps $r$ in a
longest chain of prime ideals%
\begin{equation*}
A_{0}\varsubsetneq A_{1}\varsubsetneq \cdots \varsubsetneq A_{r}
\end{equation*}%
in $R$. It is well-known that $\mathrm{sr}(R)\leq \mathrm{Kdim}(R)+1$ (cf.
Section 4E of \cite{mag}). When $R$ is a Dedekind domain, then $\mathrm{sr}%
(R)\leq 2.$ When $G$ is finite and $R$ is a Dedekind domain, the stable
range $\mathrm{sr}(R[G])\leq 2$ (cf. 41.23 of page 98 in \cite{rep}). The
stable range is not bigger than most other famous ranges, e.g. absolute
stable range, $1+$ maximal spectrum dimension, $1+$ Bass-Serre dimension
(cf. \cite{pec}) and so on.

The following result on stabilization of $K_{1}$ is Theorem 10.15 in \cite%
{mag}.

\begin{lemma}
\label{2.2}Let $R$ be a ring of finite stable range $\mathrm{sr}(R).$ Then
for an integer $n\geq \mathrm{sr}(R)+1,$ the natural map
\begin{equation*}
\mathrm{GL}_{n}(R)/E_{n}(R)\rightarrow K_{1}(R)
\end{equation*}%
is an isomorphism.
\end{lemma}

We define the unitary $K$-group $KU_{1}$ as follows. There is an obvious
embedding
\begin{equation*}
U_{2n}(R,\Lambda )\rightarrow U_{2(n+1)}(R,\Lambda ),
\end{equation*}%
\begin{equation*}
\left(
\begin{array}{ll}
\alpha & \beta \\
\gamma & \delta%
\end{array}%
\right) \longmapsto \left(
\begin{array}{llll}
\alpha & 0 & \beta & 0 \\
0 & 1 & 0 & 0 \\
\gamma & 0 & \delta & 0 \\
0 & 0 & 0 & 1%
\end{array}%
\right) .\
\end{equation*}%
Using this map, we shall consider $U_{2n}(R,\Lambda )$ as a subgroup of $%
U_{2(n+1)}(R,\Lambda )$. Similarly, define $U(R,\Lambda )$ as the direct
limit of
\begin{equation*}
U_{2}(R,\Lambda )\subset U_{4}(R,\Lambda )\subset \cdots \subset
U_{2n}(R,\Lambda )\subset \cdots
\end{equation*}%
and $E(R,\Lambda )$ as the direct limit of%
\begin{equation*}
EU_{2}(R)\subset EU_{4}(R)\subset \cdots \subset EU_{2n}(R)\subset \cdots .
\end{equation*}%
The unitary $K$-theory group $KU_{1}(R,\Lambda )$ is defined as $U(R,\Lambda
)/E(R,\Lambda ).$

The $\Lambda $\emph{-stable range condition} $\Lambda \mathrm{sr}_{m}$ of
Bak and Tang \cite{bt} says that $R$ satisfies $\mathrm{sr}_{m}$ and given
any unimodular vector $(a_{1},\ldots ,a_{m+1},b_{1},\ldots ,b_{m+1})\in
R^{2(m+1)}$ there exists a matrix $\gamma \in \Lambda _{m+1}$ such that $%
(a_{1},\ldots ,a_{m+1})+(b_{1},\ldots ,b_{m+1})\gamma $ is unimodular. By
\cite{bt}, $\Lambda \mathrm{sr}_{m}\Rightarrow \Lambda \mathrm{sr}_{n}$ for
all $n\geq m$. The $\Lambda $\emph{-stable range} $\Lambda \mathrm{sr}(R)$
of $(R,\Lambda )$ is the smallest number $m$ such that $\Lambda \mathrm{sr}%
_{m}$ holds. In general, the $\Lambda $-stable range is also not bigger than
$1+$ Bass-Serre dimension (cf. \cite{pec}).

The following result on stabilization of $KU_{1}$ was proved by Bak and Tang
in \cite{bt}.

\begin{lemma}
\label{2.4} Let $(R,\Lambda )$ be a form ring with finite $\Lambda $-stable
range $\Lambda \mathrm{sr}(R)$. Then for an integer $n\geq \Lambda \mathrm{sr%
}(R)+1$ the natural map
\begin{equation*}
U_{2n}(R,\Lambda )/EU_{2n}(R,\Lambda )\rightarrow KU_{1}(R,\Lambda )
\end{equation*}%
is an isomorphism$.$
\end{lemma}

\section{Proof of Theorems}

In this section, we prove the results presented in Section \ref{section1}.

\subsection{Group actions on \textrm{CAT(0)} spaces}

In order to prove Theorem \ref{th3}, we need the following lemma. This is a
generalization of Proposition 2 in \cite{fu}, which is stated for Chevalley
groups over commutative rings. Recall that the permutation $\sigma $ is
defined in Section \ref{sec}.

\begin{lemma}
\label{fuu}

\begin{enumerate}
\item[(i)] Let $R$ be a general ring. Then for any integer $n\geq 3,$ $1\leq
i\neq j\leq n$, an element $r\in R$ and any elementary matrix $e_{ij}(r)\in
E_{n}(R),$ there exists a nilpotent subgroup $U\subset E_{n}(R)$ such that $%
e_{ij}(r)\in \lbrack U,U]$.

\item[(ii)] Let $(R,\Lambda )$ be a form ring over a general ring $R$. Then
for any integer $n\geq 3,$ $1\leq i\neq j\leq 2n$, an element $r\in R$ and
any elementary matrix $\rho _{ij}(r)\in EU_{2n}(R,\Lambda )$ (when $i=\sigma
j,$ we assume that $r\in \Lambda $ or $\Lambda ^{\ast }$) there exists a
nilpotent subgroup $U\subset EU_{n}(R)$ such that $\rho _{ij}(r)\in \lbrack
U,U]$.
\end{enumerate}
\end{lemma}

\begin{proof}
These are easy consequences of commutator formulas. For example, we have $%
e_{12}(r)=[e_{13}(1),e_{32}(r)].$ We choose $U$ to be the subgroup generated
by all elementary matrices $e_{13}(x),e_{32}(y)$ with $x,y\in R.$ Since the commutator $[e_{13}(x), e_{32}(y)] = e_{12}(xy)$ is central in $U$, it is
clear that this is a nilpotent subgroup. Other cases are similar. For the
group $EU_{2n}(R,\Lambda )$ and $i\neq \sigma j,$ the statement for $\rho
_{ij}(r)$ is similar to that of $e_{12}(r)$ in $E_{n}(R).$ When $i=\sigma j,$
for example $\rho _{1,n+1}(r)$ with $r\in \Lambda ^{\ast },$ we have
identities
\begin{eqnarray*}
\rho _{1,n+1}(r) &=&\rho _{1,n+2}(-r)[\rho _{12}(1),\rho _{2,n+2}(r)] \\
&=&[\rho _{13}(1),\rho _{3,n+2}(-r)][\rho _{12}(1),\rho _{2,n+2}(r)]
\end{eqnarray*}%
by (4) and (2) of Lemma \ref{2.1}. Take $U$ to be the subgroup generated by
all unitary elementary matrices $\rho _{12}(x),\rho _{13}(y),\rho
_{3,n+2}(z),\rho _{2,n+2}(a)$ with $x,y,z\in R$ and $a\in \Lambda ^{\ast }.$
This is also a nilpotent group by the commutator formulas for unitary groups
in Lemma \ref{2.1}, since the commutators of these matrices are all upper triangular matrices.
\end{proof}

Our proof of Theorem \ref{th3} will be based on the following general
fixed-point theorem, which is Theorem 5.1 in Farb \cite{fa}.

\begin{lemma}
\label{general}Let $\Gamma $ be a finitely generated group, and let $%
C=\{\Gamma _{1},\Gamma _{2},\ldots ,\Gamma _{r+1}\}$ be a collection of
finitely generated nilpotent subgroups of $\Gamma .$ Suppose that:

\begin{enumerate}
\item[(1)] $C$ generates a finite index subgroup of $\Gamma .$

\item[(2)] Any proper subset of $C$ generates a nilpotent group.

\item[(3)] There exists $m>0$ so that for any element $g$ of any $\Gamma
_{i} $ there is a nilpotent subgroup $N<\Gamma $ with $g^{m}\in \lbrack
N,N]. $
\end{enumerate}

Then $\Gamma $ has the strong property $\mathrm{FA}_{r-1}.$
\end{lemma}

\begin{proof}[Proof of Theorem \protect\ref{th3}]
We first prove that the elementary group $E_{n}(R)$ has the property $%
\mathrm{FA}_{n-2}$. If a group has strong property $\mathrm{FA}_{r},$ then
so do all its quotient groups (cf. (2) of Lemma \ref{lemfarb}). Therefore,
we may assume that the ring $R$ is the free noncommutative ring $\mathbb{%
Z\langle }x_{1},x_{2},\ldots ,x_{k}\mathbb{\rangle }$ generated by elements $%
x_{1},x_{2,}\ldots ,x_{k}.$ For $1\leq i\leq n-1,$ let $\Gamma _{i}$ be the
subgroup generated by all matrices $e_{i,i+1}(x)$ with $x\in R$. Denote by $%
\Gamma _{n}$ the subgroup generated by all matrices $e_{n1}(x)$ with $x\in R$%
. Then the set $C:=\{\Gamma _{1},\Gamma _{2},\ldots ,\Gamma _{n}\}$
generates the whole group $E_{n}(R),$ as follows. Denote by $\langle
C\rangle $ the subgroup generated by $C$ in $E_{n}(R).$ By the commutator
formulas in Lemma \ref{ecom}, when $r\in R$ and $1\leq i<j\leq n,$ we have
that
\begin{eqnarray*}
e_{ij}(r)
&=&[e_{i,i+1}(r),e_{i+1,j}(1)]=[e_{i,i+1}(r),[e_{i+1,i+2}(1),e_{i+2,j}(1)]]
\\
&=&\cdots =[e_{i,i+1}(r),[\cdots ,e_{j-1,j}(1)]\cdots ]\in \langle C\rangle
\end{eqnarray*}%
and
\begin{equation*}
e_{ji}(r)=[e_{jn}(r),e_{ni}(1)]=[e_{jn}(r),[e_{n1}(1),e_{1i}(1)]]\in \langle
C\rangle .
\end{equation*}%
This checks $(1)$ of Lemma \ref{general}. It is obvious that $(2)$ also
holds. By Lemma \ref{fuu}, the condition $(3)$ holds as well for $m=1$.
Therefore, Lemma \ref{general} implies that $E_{n}(R)$ has the strong
property $\mathrm{FA}_{n-2}.$

We prove the property $\mathrm{FA}_{n-1}$ of the elementary unitary group $%
EU_{2n}(R,\Lambda )$ as follows. The idea is the same as the proof for $%
E_{n}(R).$ For $1\leq i\leq n-1,$ let $\Gamma _{i}$ be the subgroup
generated by all $\rho _{i,i+1}(x)$ with $x\in R$. Denote by $\Gamma _{n}$
the subgroup generated by all $\rho _{n,2n-1}(r)\rho _{n,2n}(x)$ with $r\in
R $, $x\in \Lambda ^{\ast }$ and by $\Gamma _{n+1}$ the subgroup generated
by all $\rho _{n+1,2}(r)\rho _{n+1,1}(x)$ with $r\in R,x\in \Lambda $. Let $%
C^{\prime }$ be the set of subgroups $\{\Gamma _{1},\Gamma _{2},\ldots
,\Gamma _{n+1}\}.$ It is sufficient to check that all the conditions in
Lemma \ref{general} are satisfied. By Lemma \ref{fuu}, for any integer $%
1\leq i\leq n-1,$ the group $\Gamma _{i}$ satisfies the condition (3). Note
that for any $r\in R$, $x\in \Lambda ^{\ast },$ by Lemma \ref{2.1} we have
that
\begin{eqnarray*}
\rho _{n,2n-1}(r)\rho _{n,2n}(x) &=&\rho _{n,2n-1}(r-x)[\rho
_{n,n-1}(1),\rho _{n-1,2n-1}(x)] \\
&=&[\rho _{n1}(r-x),\rho _{1,2n-1}(1)][\rho _{n,n-1}(1),\rho _{n-1,2n-1}(x)].
\end{eqnarray*}%
Therefore, any element of the group $\Gamma _{n}$ lies in the commutator
subgroup of the nilpotent subgroup generated by all matrices $\rho
_{n1}(r_{1}),$ $\rho _{n,n-1}(r_{2}),\rho _{1,2n-1}(r_{3})$ and $\rho
_{n-1,2n-1}(x)$ with $r_{1},r_{2},r_{3}\in R$ and $x\in \Lambda ^{\ast }.$ A
similar argument shows that $\Gamma _{n+1}$ satisfies the condition (3) as
well. We now check the condition (1). Denote by $\langle C^{\prime }\rangle $
the subgroup generated by $C^{\prime }$ in $EU_{2n}(R,\Lambda ).$ According
to the commutator formulas in Lemma \ref{2.1}, for any $r\in R$ and $1\leq
i<j\leq n$ we have that $\rho _{ij}(r)\in \langle C^{\prime }\rangle $ and%
\begin{equation*}
\rho _{i,2n}(r)=[\rho _{i,n-1}(1),\rho _{n-1,2n}(r)]\in \langle C^{\prime
}\rangle .
\end{equation*}%
Note that $\rho _{i,2n}(r)=\rho _{n,n+i}(-\varepsilon ^{\ast }r^{\ast }).$
When $1\leq i<n<j\leq 2n$ with $i\neq n-j$ and $r\in R,$ we have that
\begin{equation*}
\rho _{ij}(r)=[\rho _{in}(1),\rho _{nj}(r)]\in \langle C^{\prime }\rangle .
\end{equation*}%
Since all the matrices $\rho _{i,\sigma i}(x)$ can be generated by $\rho
_{ij}(1)$ with $i\neq \sigma j$ and $\rho _{n,2n}(x)$ (cf. (4) in Lemma \ref%
{2.1}), we get that all the upper triangular elementary unitary matrices
belong to $\langle C^{\prime }\rangle $. For any $r\in R,1<i\leq 2n$ with $%
i\neq n+1,n+2,$ we have that
\begin{equation*}
\rho _{n+1,i}(r)=[\rho _{n+1,2}(r),\rho _{2,i}(1)]\in \langle C^{\prime
}\rangle .
\end{equation*}%
Note that for any $r\in R$ and $i\neq 1,n+1,$ the matrix $\rho
_{i,1}(r)=\rho _{n+1,\sigma i}(x)$ for some $x\in R.$ Therefore for any $%
r\in R$ and all $1<i,j\leq 2n$ with $i,j,\sigma i,\sigma j$ distinct and $%
i,j\neq n+1$, we have
\begin{equation*}
\rho _{ij}(r)=[\rho _{i1}(1),\rho _{1j}(r)].
\end{equation*}%
This proves that the subgroup generated by $C^{\prime }$ is $%
EU_{2n}(R,\Lambda )$ and the condition (1) in Lemma \ref{general} is
satisfied. It can be directly checked that condition (2) holds. Therefore,
the group $EU_{2n}(R,\Lambda )$ has property $\mathrm{FA}_{n-1}$ by Lemma %
\ref{general}.
\end{proof}

\bigskip

\begin{proof}[Proof of Theorem \protect\ref{th4}]
Recall the stabilization of $K_{1}$ from Lemma \ref{2.2}. When $n\geq
\mathrm{sr}(R)+1,$ the group $E_{n}(R)$ is normal in $\mathrm{GL}_{n}(R)$
and there is an isomorphism $\mathrm{GL}_{n}(R)/E_{n}(R)\rightarrow
K_{1}(R). $ When $n\geq \max \{3,\mathrm{sr}(R)+1\},$ the group $E_{n}(R)$
has property $\mathrm{FA}_{n-2}$ by Theorem \ref{th3}. By assumption, the
quotient group $\mathrm{GL}_{n}(R)/E_{n}(R)\cong K_{1}(R)$ has property $%
\mathrm{FA}_{n-2}.$ Therefore, the group $\mathrm{GL}_{n}(R)$ has property $%
\mathrm{FA}_{n-2}$ according to (3) of Lemma \ref{lemfarb}. The second part
for $U_{2n}(R,\Lambda )$ can be proved similarly using Lemma \ref{2.4} and
Theorem \ref{th3}.
\end{proof}

\bigskip

\begin{proof}[Proof of Corollary \protect\ref{cor5}]
When $G$ is finite and $R$ is a Dedekind domain, the stable range $\mathrm{sr%
}(R[G])\leq 2$ (41.23 of \cite{rep}, page 98). When $R=\mathbb{Z}$, the
abelian group $K_{1}(\mathbb{Z[}G])$ is finitely generated of rank equal to
the number of irreducible real representations of $G$ minus the number of
irreducible rational representations (Theorem 7.5 of \cite{Bass}, page 625).
By assumption, we have that the group $K_{1}(\mathbb{Z}[G])$ is finite. By
Lemma 2.1 in \cite{fa}, any finite group action on a \textrm{CAT(0)} space
has a global fixed point and thus has property $\mathrm{FA}_{n-2}.$ This
finishes the proof by Theorem \ref{th4}.
\end{proof}

\bigskip

In order to prove Proposition \ref{propcat}, we need the following lemma,
which was pointed out to the author by A.J. Berrick. This is a
generalization of Lemma 4.2 in \cite{ck} which is stated for $R=\mathbb{Z}$.

\begin{lemma}
\label{lem4.2}Let $n\geq 3$ and $R$ a general ring. Then any action of $%
E_{n}(R)$ on a finite set with less than $n$ points is trivial.
\end{lemma}

\begin{proof}
Let \textrm{Sym}$(k)$\textrm{\ }be the permutation group of $k$ elements.
Any group action of $E_{n}(R)$ on a finite set of $k$ elements corresponds a
group homomorphism%
\begin{equation*}
\varphi :E_{n}(R)\rightarrow \mathrm{Sym}(k)\mathrm{.}
\end{equation*}%
When $k\leq n-1$ and $n\geq 5,$ the alternating group $A_{n}$ is simple and
there is no nontrivial map from $A_{n}$ to $\mathrm{Sym}(k)$ by considering
the cardinalities. Since $A_{n}$ normally generates $E_{n}(R)$ (cf. Berrick
\cite{ber}, 9.4), any map $\varphi $ is trivial. For $n=3$ and $n=4,$ the
triviality of $\varphi $ follows from the fact that $E_{n}(R)$ is perfect
and $\mathrm{Sym}(k)$ is soluble.
\end{proof}

\bigskip

\begin{proof}[Proof of Proposition \protect\ref{propcat}]
Let $X$ be a uniformly finite \textrm{CAT(0)} cell complex. Assume that the
degree of each vertex is less $N$ for some positive integer $N.$ For an
integer $n\geq \max \{\dim (X)-2,3\},$ let $G$ be a copy of $E_{n}(R)$
sitting inside of $E(R)$ (or inside of $EU(R,\Lambda )$ by the hyperbolic
embedding defined by $A\longmapsto \mathrm{diag}(A,A^{\ast -1})$). We may
assume that $n>N.$ By Theorem \ref{th3}, there is a fixed point $x_{0}\in X$
under the $G$-action. Denote by $\mathrm{Fix}(G)$ the set of fixed points of
$G$-action in $X.$ Then $G$ acts on the link of $x_{0},$ which is a finite
set with less than $N$ elements. By Lemma \ref{lem4.2}, the group $G$ action
is trivial, which shows that any neighbor of $x_{0}$ is also in $\mathrm{Fix}%
(G).$ Therefore, the group $G$ acts trivially on all vertices of $X.$
According to the commutator formulas in Lemma \ref{ecom} and Lemma \ref{2.4}%
, the group $E(R)$ and $EU(R,\Lambda )$ are normally generated by $G$ and
hence act trivially on the whole space $X.$
\end{proof}

\bigskip

\begin{proof}[Proof of Corollary \protect\ref{correp}]
It is proved by Farb in Theorem 1.7 and Theorem 1.8 of \cite{fa} that any
group $\Gamma $ with property $\mathrm{FA}_{n-1}$ is of integral $n$%
-representation type. Then the corollary is a direct consequence of Theorem %
\ref{th3}.
\end{proof}

\subsection{Group actions on spheres and acyclic manifolds}

Recall that a group $G$ action on a space $X$ is effective if the subgroup
that fixes all elements of $X$ is trivial. In order to prove Theorem \ref%
{th1}, we need two lemmas from Bridson and Vogtmann \cite{BV}.

\begin{lemma}[ \protect\cite{BV}, Theorem 4.7]
\label{lemm1}Let $m$ and $n$ be two integers with $m<n-1.$ Then the group $%
\mathbb{Z}_{2}^{n}$, $n$ copies of groups of two elements, cannot act
effectively by homeomorphisms on a generalized $m$-sphere over $\mathbb{Z}%
_{2}$ or a $\mathbb{Z}_{2}$-acyclic $(m+1)$-hm$_{\mathbb{Z}_{2}}.$

If $m<2n-1$ and $p$ is an odd prime, then $\mathbb{Z}_{p}^{n}$ cannot act
effectively by homeomorphisms on a generalized $m$-sphere over $\mathbb{Z}%
_{p}$ or a $\mathbb{Z}_{p}$-acyclic $(m+1)$-hm$_{\mathbb{Z}_{p}}.$
\end{lemma}

\begin{lemma}[\protect\cite{BV}, Lemma 4.12]
\label{lemm2}Let $X$ be a generalized $m$-sphere over $\mathbb{Z}_{2}$ or a $%
\mathbb{Z}_{2}$-acyclic $(m+1)$-hm$_{\mathbb{Z}_{2}}$ and $G$ be a group
acting by homeomorphisms on $X$. Suppose $G$ contains a subgroup $P=\mathbb{Z%
}_{2}\times \mathbb{Z}_{2}$ all of whose nontrivial elements are conjugate
in $G$. If $P$ acts nontrivially, then the fixed-point sets of its
nontrivial elements have codimension $m\geq 2.$
\end{lemma}

\begin{proof}[Proof of Theorem \protect\ref{th1}]
We only give the proof of group actions on generalized $d$-spheres, while
that of group actions on acyclic homology manifolds is similar. Suppose that
$E_{n}(R)$ acts by homeomorphisms on some generalized $d$-sphere $X$. This
means that there is a group homomorphism $f:E_{n}(R)\rightarrow \mathrm{Homeo%
}(X).$ We prove (a)(i) in two cases.

\begin{enumerate}
\item[(1)] \noindent The characteristic of $R$ is $2.$
\end{enumerate}

When $n=3,$ the elements $e_{12}(1),e_{13}(1)$ generate a subgroup which is
isomorphic to $G:=\mathbb{Z}_{2}^{2}$ in $E_{n}(R).$ Note that $e_{12}(1)$
and $e_{13}(1)$ are conjugate by a permutation matrix and that
\begin{equation*}
e_{12}(1)e_{13}(1)=e_{23}(1)e_{12}(1)e_{23}(1).
\end{equation*}%
We conclude from Lemma \ref{lemm2} that if the action of $G$ is not trivial
then the fixed-point set of any nontrivial element is at least of
codimension $2$. Since $d\leq 1,$ this shows that the action of $G$ is free.
However, a classical result of Smith says that $\mathbb{Z}_{p}\times \mathbb{%
Z}_{p}$ cannot act freely on any generalized sphere over $\mathbb{Z}_{p}$
for any prime number $p$ (cf. \cite{s}). This implies that the action of $G$
is trivial. By the commutator formulas in Lemma \ref{ecom}, the group $%
E_{n}(R)$ is normally generated by $G$. This shows that the action of $%
E_{n}(R)$ is trivial. When $n\geq 4,$ the matrices $e_{ij}(1)$ $(1\leq i\leq
n/2,$ $n/2<j\leq n)$ generate an abelian group $\mathbb{Z}_{2}^{k},$ where
in general $k\geq n.$ By Lemma \ref{lemm1}, the action of $\mathbb{Z}%
_{2}^{k} $ is not effective on the generalized $d$-sphere $X$ over $\mathbb{Z%
}_{2}$. Choose a nontrivial element $M\in \mathbb{Z}_{2}^{k}$ acting
trivially on $X$. Without loss of generality, we may assume that $%
M=e_{1n}(1).$ By the commutator formulas in Lemma \ref{ecom} again, the
group $E_{n}(R)$ is normally generated by such $M$. This shows that the
action of $E_{n}(R)$ is trivial. The same argument with $x_{ij}(1)$ instead
of $e_{ij}(1)$ show that any action of $\mathrm{St}_{n}(R)$ on $X$ is also
trivial.

\begin{enumerate}
\item[(2)] \noindent The characteristic of $R$ is not $2.$
\end{enumerate}

Let $A_{12},A_{23},\ldots ,A_{n-1,n}$ be the elements in $E_{n}(R)$ defined
in Lemma \ref{prop}. By Lemma \ref{prop}, they generate a subgroup which is
isomorphic to $\mathbb{Z}_{2}^{n-1}.$ Suppose that we can find a noncentral
element $A$ of $E_{n}(R)$ in $\mathbb{Z}_{2}^{n-1}$ such that the action of $%
A$ is trivial. According to Theorem \ref{normal}, the normal subgroup
generated by $A$ contains the subgroup $E_{n}(R,2R).$ Note that the action
of any element in $E_{n}(R,2R)$ is trivial. When $2$ is invertible in $R,$
we have that $E_{n}(R,2R)=E_{n}(R).$ This implies that any element in $%
E_{n}(R)$ acts trivially on $X$. When $2$ is not invertible, the action of $%
E_{n}(R)$ factors through that of $E_{n}(R)/E_{n}(R,2R).$ Note that there is
a commutative diagram%
\begin{equation*}
\begin{array}{ccccc}
1\rightarrow \mathrm{St}_{n}(R,2R) & \rightarrow & \mathrm{St}_{n}(R) &
\rightarrow & \mathrm{St}_{n}(R/2R)\rightarrow 1 \\
\downarrow &  & \downarrow &  & \downarrow \\
1\rightarrow E_{n}(R,2R) & \rightarrow & E_{n}(R) & \rightarrow &
E_{n}(R)/E_{n}(R,2R)\rightarrow 1,%
\end{array}%
\end{equation*}%
where the two horizontal sequences are exact (for the exactness of the first
one, see Lemma 13.18 and its proof in Magurn \cite{mag}). Then the action of
$E_{n}(R)/E_{n}(R,2R)$ on $X$ can be lifted as an action of $\mathrm{St}%
_{n}(R/2R).$ Since the quotient ring $R/2R$ is of characteristic $2,$ this
case is already proved in case $(1).$ Therefore, it is enough to find such
element $A$ in $\mathbb{Z}_{2}^{n-1}$ such that the action of $A$ is
trivial. It is not hard to see that for each integer $1\leq i\leq n-2,$ the
elements $A_{i,i+1},$ $A_{i+1,i+2}$ and $A_{i,i+1}A_{i+1,i+2}$ are conjugate
by some permutation matrices. We will finish the proof by induction on $n$
(cf. the proof of Theorem 1.1 in Bridson and Vogtmann \cite{BV}).

When $n=3,$ using a similar argument as that of case (1), we see that the
group generated by $A_{12}$ and $A_{23}$ cannot act effectively on the
generalized $d$-sphere $X.$ Therefore such element $A$ exists.

When $n=4,$ if the action of $A_{12}$ is trivial, we are done. Otherwise,
Lemma \ref{lemm2} and Lemma \ref{locsmith} show that the fixed-point set $%
\mathrm{Fix}(A_{12})$ of $A_{12}$ is a generalized sphere over $\mathbb{Z}%
_{2}$ of dimension $0$ (note: the fixed-point set is not empty). Then the
abelian group $\mathbb{Z}_{2}^{2}$ generated by $A_{23}$ and $A_{34}$ acts
on $\mathrm{Fix}(A_{12}).$ By Lemma \ref{lemm1}, there exists a nontrivial
element $\gamma $ with trivial action on $\mathrm{Fix}(A_{12}).$ Since $%
\gamma $ and $A_{12}$ are conjugate, we have that $\mathrm{Fix}(A_{12})=%
\mathrm{Fix}(\gamma ).$ By Theorem 4.8 in \cite{BV}, $A_{12}$ and $\gamma $
have the same image in $\mathrm{Homeo}(X).$ If $\gamma \neq A_{34},$ $%
A_{12}\gamma ^{-1}$ is noncentral in $E_{n}(R)$ and we can take $%
A=A_{12}\gamma ^{-1}.$ If $\gamma =A_{34},$ the group homomorphism $f$
factors through%
\begin{equation*}
\bar{f}:E_{n}(R)/\langle \pm I_{n}\rangle \rightarrow \mathrm{Homeo}(X).
\end{equation*}%
In $E_{n}(R)/\langle \pm I_{n}\rangle ,$ the images of $A_{12},A_{23},$ $%
e_{12}(1)e_{21}(-1)e_{12}(1)e_{34}(1)e_{43}(-1)e_{34}(1)$ and $%
e_{13}(1)e_{31}(-1)e_{13}(1)e_{24}(1)e_{42}(-1)e_{24}(1)$ generate an
abelian group $\mathbb{Z}_{2}^{4}.$ By Lemma \ref{lemm1}, there exists a
nontrivial element having trivial action on $X.$ The preimage of such an
element normally generates $E_{n}(R,2R)$. By case (1), we are done.

We now consider the general case when $n\geq 5.$ If the action of $A_{n-1,n}$
is trivial, we are done. Otherwise, Lemma \ref{lemm2} and Lemma \ref%
{locsmith} show that the fixed-point set $\mathrm{Fix}(A_{n-1,n})$ is a
generalized sphere over $\mathbb{Z}_{2}$ of codimension at least $2.$ The
elements in the subgroup $E_{n-2}(R)$ in the upper left corner of $E_{n}(R)$
are centralizers of $A_{n-1,n}.$ By induction assumption, the action of $%
E_{n-2}(R)$ on $\mathrm{Fix}(A_{n-1,n})$ is trivial. This shows that $%
\mathrm{Fix}(A_{n-1,n})\subset \mathrm{Fix}(A_{12}).$ Similarly, the
converse holds. This implies that $f(A_{12})=f(A_{n-1,n})$ (cf. Theorem 4.8
in \cite{BV}). Take $A=A_{12}^{-1}A_{n-1,n}.$ This finishes the proof of (a).

We prove (b)(i) as follows. Since $E_{n-1}(R)$ normally generates $E_{n}(R)$
when $n>2,$ it is enough to prove (ii) when $n=2k$ for some $k\geq 2$.
Construct an abelian subgroup $\mathbb{Z}_{3}^{k}$ in $E_{n}(R)$, as
follows. For each integer $i$ $(i=1,2,\ldots ,k),$ denote by $B_{i}$ the
matrix
\begin{equation*}
e_{2i-1,2i}(1)e_{2i,2i-1}(-1)e_{2i-1,2i}(1)e_{2i,2i-1}(-1)\in E_{n}(R).
\end{equation*}%
For example, $B_{1}$ looks like the matrix
\begin{equation*}
\begin{pmatrix}
-1 & 1 &  \\
-1 & 0 &  \\
&  & I_{n-2}%
\end{pmatrix}%
.
\end{equation*}%
It is obvious that each matrix $B_{i}$ has order $3$ and together they
generate an abelian subgroup $\mathbb{Z}_{3}^{k}$ in $E_{n}(R).$ By Lemma %
\ref{lemm1}, for an integer $d\leq 2k-2$ the group $\mathbb{Z}_{3}^{k}$
cannot act effectively by homeomorphisms on a generalized $d$-sphere over $%
\mathbb{Z}_{3}.$ Without loss of generality, we may assume that the action
of $B_{1}$ is trivial. Note that
\begin{equation*}
\lbrack e_{32}(1),B_{1}]=e_{31}(-1)e_{32}(2)
\end{equation*}%
and
\begin{equation*}
\lbrack e_{31}(-1)e_{32}(2),e_{12}(-1)]=e_{32}(1).
\end{equation*}%
The matrix $e_{32}(1)$ normally generates the whole group $E_{n}(R).$ This
shows that the group action of $E_{n}(R)$ is trivial.

Now we prove (c). Suppose that the group $EU_{2n}(R,\Lambda )$ acts by
homeomorphisms on a generalized $d$-homology sphere over $\mathbb{Z}_{2}$ or
$\mathbb{Z}_{3}$. There is a group homomorphism $E_{n}(R)\rightarrow
EU_{2n}(R,\Lambda )$ defined by the hyperbolic embedding
\begin{equation*}
A\longmapsto \mathrm{diag}(A,A^{\ast -1})
\end{equation*}%
for any element $A\in E_{n}(R).$ By the commutator formulas in Lemma \ref%
{2.1}, we see that $EU_{2n}(R,\Lambda )$ is normally generated by the image
of $E_{n}(R).$ Since the action of $E_{n}(R)$ is trivial, the action of $%
EU_{2n}(R,\Lambda )$ is trivial as well.
\end{proof}

\begin{remark}
\emph{If the generalized spheres in Theorem \ref{th1} are smooth manifolds
and the actions are smooth, the proof is much easier by noting the fact that
}$\mathbb{Z}^{k}$\emph{\ cannot act effectively by orientation-preserving
diffeomorphisms on a }$d$\emph{-sphere for }$d\leq k-1$\emph{\ (cf. the
proof of Theorem 2.1 in \cite{BV}). When we know that Theorem \ref{th1} is
true for }$R=\mathbb{Z}$\emph{, the general-ring case can also be proved by
using the normal generation of }$E_{n}(R)$\emph{\ by the image of }$E_{n}(%
\mathbb{Z}).$\emph{\ Our intent here is to avoid the Margulis finiteness
theorem. Moreover, the proof given here works for Steinberg groups as well.}
\end{remark}

\begin{proof}[Proof of Theorem \protect\ref{th2}]
The strategy of the proof is similar to that of Theorem \ref{th1}. We
construct an abelian subgroup $\mathbb{Z}_{3}^{n}$ of $EU_{2n}(R,\Lambda )$
as follows. For $i=1,2,\ldots ,n,$ let
\begin{equation*}
C_{i}=\rho _{i,n+i}(1)\rho _{n+i,i}(-1)\rho _{i,n+i}(1)\rho _{n+i,i}(-1)\in
EU_{2n}(R).
\end{equation*}%
It is obvious that the order of $C_{i}$ is $3$ and the subgroup generated by
$C_{i}$ $(i=1,2,\ldots ,n)$ is $\mathbb{Z}_{3}^{n}.$ The remainder of the
proof of (i) is the same as that of (b)(i) in Theorem \ref{th1}.
\end{proof}

\bigskip

\begin{proof}[Proof of Corollary \protect\ref{repcor}]
Let $E_{n-1}(A)$ act on the space $\mathbb{R}^{n-1}$ by matrix
multiplications. According to Theorem \ref{th1} a(ii)\emph{, }the image of $%
E_{n}(R)$ in $E_{n-1}(A)$ acts trivially on $\mathbb{R}^{n-1}$. This implies
that the image in (i) is the identity matrix. The second part can be proved
similarly by using Theorem \ref{th2} and considering the group $\mathrm{Sp}%
_{2(n-1)}(A)$ action on the space $\mathbb{R}^{2(n-1)}$.
\end{proof}

\bigskip

\begin{proof}[Proof of Theorem \protect\ref{thcor}]
For the group $E(R),$ the proof is similar to that of Lemma 1 in \cite{sw}.
The idea is as follows. For sufficiently large $k,$ the abelian group $%
\mathbb{Z}_{2}^{k}$ cannot act effectively on the manifolds in Theorem \ref%
{thcor}. When the characteristic of $R$ is $2,$ we take such $\mathbb{Z}%
_{2}^{k}$ as the subgroup in $E(R)$ generated by $e_{1j}(1)$ for $2\leq
j\leq k+1.$ By commutator formulas (cf. Lemma \ref{ecom}), any nontrivial
element in $\mathbb{Z}_{2}^{k}$ normally generates $E(R).$ This shows that
the action of $E(R)$ is trivial. When the characteristic of $R$ is not $2,$
we take such $\mathbb{Z}_{2}^{k}$ as the subgroup generated by $A_{i,i+1}$
defined in Lemma \ref{prop} for $1\leq i\leq k.$ Any nontrivial element in
such $\mathbb{Z}_{2}^{k}$ is noncentral in $E(R).$ By Lemma \ref{prop}, any
noncentral element in such $\mathbb{Z}_{2}^{k}$ generates a normal subgroup
containing $E(R,2R).$ Therefore the action of $E(R)$ factors through that of
$E(R/2R),$ which is already proved since the characteristic of $R/2R$ is $2$.

For the group $EU(R,\Lambda ),$ note that there is a hyperbolic embedding $%
E(R)\rightarrow EU(R,\Lambda )$ defined by $A$ $\longmapsto \mathrm{diag}%
(A,A^{\ast -1}).$ The action of $EU(R,\Lambda )$ is trivial since $E(R)$
normally generates $EU(R,\Lambda ).$
\end{proof}

\bigskip

\noindent \textbf{Acknowledgements}

The author wants to thank Professor Wolfgang L\"{u}ck for supporting from
his Leibniz-Preis a visit to Hausdorff Center of Mathematics in University
of Bonn from April 2011 to July 2011, when parts of this paper were written.
He is also grateful to his advisor Professor A. J. Berrick for many helpful
discussions and to Professor Martin Kassabov who carefully read the
manuscript and made many useful comments.

\bigskip

Department of Mathematics, National University of Singapore, Kent Ridge
119076, Singapore.

E-mail: yeshengkui@nus.edu.sg
\end{document}